\documentclass[10pt, leqno]{article}
\usepackage{geometry}           
\usepackage{amsmath, amsthm, amssymb, amscd, tikz-cd, physics, enumitem, titlesec, titletoc}
\usepackage[mathscr]{euscript}

\newcommand{\id}{\mathrm{id}}
\renewcommand{\d}{\partial}

\newcommand{\R}{\mathbb{R}}

\newcommand{\til}[1]{\tilde{#1}}

\renewcommand{\div}{\mathrm{div}}
\newcommand{\Ric}{\mathrm{Ric}}

\newcommand{\scal}{\mathrm{scal}}

\newcommand{\PIC}{\mathrm{PIC}}

\makeatletter
\newcommand\incircbin
{%
  \mathpalette\@incircbin
}
\newcommand\@incircbin[2]
{%
  \mathbin%
  {%
    \ooalign{\hidewidth$#1#2$\hidewidth\crcr$#1\bigcirc$}%
  }%
}
\newcommand{\owedge}{\incircbin{\land}}
\makeatother

\newtheorem{theorem}{Theorem}[section]
\newtheorem{corollary}[theorem]{Corollary}
\newtheorem{lemma}[theorem]{Lemma}

\theoremstyle{definition}

\numberwithin{equation}{section}

\titleformat{\section}[block]{\scshape\filcenter}{\thesection.}{3pt}{}
\titleformat{\subsection}[block]{\scshape}{\thesubsection.}{3pt}{}
\titlecontents{section}[0pt]{}{\thecontentslabel. \,}{}{\hfill\contentspage}

\title{\textbf{\large{SHRINKING RICCI SOLITONS WITH POSITIVE ISOTROPIC CURVATURE}}}
\author{\textsc{\small KEATON NAFF}}
\date{} 

\begin{document}
\maketitle

\begin{abstract}
We show that in dimensions $n \geq 12$, a non-flat complete gradient shrinking solitons with uniformly positive isotropic curvature (PIC) must be a quotient of either the round sphere $S^n$ or the cylinder $S^{n-1} \times \R$. We also observe that in dimensions $n \geq 5$, a complete gradient shrinking soliton that is strictly PIC and weakly PIC2 must be a quotient of either the round sphere $S^n$ or the cylinder $S^{n-1} \times \R$.
\end{abstract}

\section{Introduction}
 
A Ricci gradient shrinking soliton is a triple $(M, g, f)$ consisting of a complete Riemannian manifold $(M, g)$ a function $f \in C^{\infty}(M)$, called the Ricci potential, satisfying the equation
\begin{equation*}
\Ric + \nabla^2 f = \frac{1}{2} g.
\end{equation*}
We say a gradient shrinking soliton is \textit{complete} if the vector field $\nabla f$ is a complete vector field. In this note, we will always assume the gradient shrinking soliton is complete.

Ricci shrinking solitons have played an important role in the analysis of singularities in the Ricci flow. By work of Perelman, under certain preserved curvature positivity conditions, the blow-up limits at a singular time of the Ricci flow have an asymptotic soliton at $t = - \infty$. The classification of these asymptotic solitons can be used to understand and eventually classify the blow-up models. 

In the Ricci flow, the curvature positivity of the blow-up is usually better than the original manifold and in classification results one typically supposes a curvature positivity that is known to hold for blow-ups and blow-downs. In three dimensions, there is of course the celebrated Hamilton-Ivey result and Perelman's original classification of noncollapsed gradient shrinking solitons with bounded nonnegative sectional curvature. Naber later proved a universal noncollapsing result for gradient shrinking solitons in all dimensions in \cite{Naber}.

In a seminal paper, Hamilton \cite{Ham97} initiated and developed the study of the Ricci flow for initial data with positive isotropic curvature (PIC for short) in four dimensions. The PIC condition was introduced by Micallef and Moore \cite{MM}. Hamilton showed that this condition is preserved by the Ricci flow (in four dimensions), constructed a continuous family of preserved curvature cones emanating from the PIC cone, and used this family to get pinching estimates for the understanding and classification of singularity models. Using this family of cones, Li, Ni, Wallach, and Wang recently completed the classification of shrinking solitons which are strictly PIC; see \cite{NiWallach}, \cite{NiWallach2}, \cite{Ni}.

In higher dimensions, the Ricci flow has been studied under the assumption of nonnegative curvature operator and recently, for $n \geq 12$, for PIC initial data by Brendle \cite{BrePIC}. For the classification of gradient shrinking solitons in higher dimensions, Munteanu and Wang \cite{MW} gave an elegant proof using the maximum principle that gradient shrinking solitons with nonnegative sectional curvature and positive Ricci curvature must be compact. Combining this result with the convergence result of Bohm and Wilking \cite{BW2} classifies, in all dimensions, gradient shrinking solitons with nonnegative curvature operator. 

A theory for PIC initial data is a major improvement over initial data with nonnegative curvature operator. The preservation of the PIC condition for $n \geq 5$ was first shown by Brendle and Schoen \cite{BS} and independently by Nguyen \cite{Ngu}. In the setting of \cite{BrePIC}, blow-up limits are uniformly PIC and weakly PIC2 (see section 3 for definition). So in this short note, our first observation is that the theorem of Munteanu and Wang can be extended from initial data with nonnegative curvature operator to initial data in the PIC2 cone. By combining their result with the convergence result of Brendle and Schoen \cite{BS} (later improved in \cite{BrendleConv}), we can replace the assumption that the curvature operator is nonnegative with the assumption that it be weakly PIC2. 

\begin{theorem} \label{thm1}
Suppose $n \geq 5$ and $(M, g, f)$ is an $n$-dimensional complete gradient shrinking soliton with curvature tensor that is strictly PIC and weakly PIC2. Then $(M, g)$ is isometric to a quotient of either a round sphere $S^n$ or a cylinder $S^{n-1} \times \R$.
\end{theorem}

An immediate corollary of the theorem above, is a classification of the asymptotic soliton of ancient $\kappa$-solutions. Following \cite{BrePIC}, we define an ancient $\kappa$-solution to be a non-flat, ancient solution to the Ricci flow of dimension $n$ which is complete, has bounded curvature, is weakly $\PIC2$, and is $\kappa$-noncollapsed on all scales. Given such an object, $(M, g(t))$, $t \in (- \infty, 0]$, for any fixed point $q \in M$ and any sequence of times $t_k \to -\infty$, there exist points $p_k \in M$ such that the reduced distance satisfies $\ell(p_k, t_k) \leq n$. By work of Perelman \cite{Per1}, the sequence of rescalings by $|t_k|^{-1}$ about the points $(p_k, t_k)$ converges to a non-flat gradient shrinking soliton. Any such limit is called the asymptotic soliton. 

\begin{corollary}
Let $n \geq 5$. Suppose $(M, g(t))$, $t \in (-\infty, 0]$ is an ancient $\kappa$-solution, which in addition is uniformly $\PIC$. Then its asymptotic soliton is either a quotient of a round sphere $S^n$ or a cylinder $S^{n-1} \times \R$.
\end{corollary}

This can be used to give an alternative proof of the universal non-collapsing of ancient $\kappa$-solutions as shown in Theorem 6.19 in \cite{BrePIC}. 

Very recently, Li and Ni gave a classification of weakly PIC1 Ricci shrinking solitons \cite{LiNi} (thereby showing PIC1 shrinkers are PIC2). Their results independently include Theorem 1.1, its corollary, and Lemma 7.2 below. Our techniques are based on those in \cite{BrePIC}.

In an effort to understand solitons only under the PIC assumption, our second theorem and the main effort of this note is a pinching condition for complete gradient shrinking solitons in all dimensions. This is similar to a result for ancient solutions proved in \cite{BHS}. For definitions of $\mathscr C_B(\R^n)$ and property $(\ast)$, see section 3. 

\begin{theorem}\label{cones}
Let $(M, g, f)$ be an $n$-dimensional complete gradient shrinking Ricci soliton and $C(s)_{s \in [0, 1]}$ a continuously varying family of cones in $\mathscr{C}_B(\R^n)$ each satisfying property $(\ast)$. Suppose that the curvature tensor of $(M, g)$ is contained in the cone $C(0)$ at every point in $M$. Then the curvature tensor of $(M, g)$ is contained in the cone $C(1)$ at every point in $M$. 
\end{theorem}

Finally, in \cite{BrePIC} Brendle constructs a new continuous family of preserved cones that pinch towards $\PIC1$ for $n \geq 12$. Using this family of preserved cones constructed, we can improve Theorem \ref{thm1} by dropping the assumption of weakly PIC2 in dimensions $n \geq 12$. 

\begin{theorem}
Suppose $n \geq 12$ and $(M, g, f)$ is a non-flat $n$-dimensional complete gradient shrinking soliton with curvature tensor that is uniformly PIC. Then $(M, g)$ is isometric to quotient of either a round sphere $S^n$ or a cylinder $S^{n-1} \times \R$. 
\end{theorem}

It is an interesting question if one can weaken the assumption of uniformly $\PIC$ to strictly $\PIC$ in the theorem above. This would give a classification result which agrees with the stronger classification currently available in dimension four.

The structure of this note is as follows. In sections 2 and 3, we review the relevant definitions, results, and conventions necessary for our results . In section 4, we give proofs of some auxiliary results needed. In sections 5, 6, and 7 we give the proofs of Theorems 1.1, 1.3, and 1.4 respectively. \\

\textbf{Acknowledgements.} The author would like to thank his advisor Simon Brendle for suggesting this question and for his encouragement and guidance. 

\section{Gradient Shrinking Solitons} 
We begin with a brief review of some of the consequences of the soliton equation. There are standard identities 
\begin{gather*}
\scal + \Delta f = \frac{n}{2}, \\
\Ric(\nabla f) = \div (\Ric )=  \frac{1}{2} \nabla \scal, \\
d(\scal + |\nabla f|^2 - f) = 0.
\end{gather*} 
The first identity follows from tracing the soliton equation. The second and third identities follow from taking the divergence of the soliton equation and applying the second Bianchi identity. In particular, after possibly shifting $f$ by a constant, we may assume that 
\begin{equation*}\label{sol_eq}
\scal + |\nabla f|^2 = f. 
\end{equation*}
A Ricci gradient shrinking soliton $(M, g, f)$ is called \textit{normalized} if the equation above holds. When $(M, g, f)$ is normalized and has nonnegative scalar curvature (for example if we assume PIC), this implies useful estimates $0 \leq \scal \leq f$ and $|\nabla f|^2 \leq f$.

We next recall the soliton potential growth-estimate due to Cao and Zhou in \cite{CZ}. They have shown that for any normalized complete gradient shrinking soliton $(M, g, f)$, there exists constants $c_1$, $c_2$, and $r_0$ depending on $n$ and the geometry of of $g$ in a ball such that 
\begin{equation*}
\frac{1}{4}\big(r(x) -c_1)^2 \leq f(x) \leq\frac{1}{4} \big(r(x) + c_2)^2
\end{equation*}
for $r(x) \geq r_0$, where $r(x) = d(x,p)$ is the distance function from a point which minimizes $f$. In particular $f(x) \to \infty$ as $d(x,p) \to \infty$.

Finally, we recall the following standard elliptic evolution equations for gradient shrinking Ricci solitons, 
\begin{align*}
\Delta_f R &= R - Q(R), \\
\Delta_f \Ric_{ik} &= \Ric_{ik} - 2\sum_{p, q =1}^nR_{ipkq} \Ric_{pq}, \\
\Delta_f \scal &= \scal - 2|\Ric|^2.
\end{align*}
Here $\Delta_f = \Delta - \nabla_{\nabla f}$. These identities follow from the standard parabolic evolutions equations for the associated Ricci flow of $(M, g, f)$ given by 
\[
g(t) = (-t) \varphi^{\ast}_{-\log(-t)} g
\]
where $\varphi_t$ is the 1-parameter flow for the vector field $\nabla f$. Here $Q(R)$ is given by
\begin{equation*}
Q(R)_{ijkl} = \sum_{p,q =1}^nR_{ijpq}R_{klpq} + 2\sum_{p,q =1}^n R_{ipkq}R_{jplq} - R_{iplq}R_{jpkq}.
\end{equation*}

\section{Preserved Cone Conditions in Ricci Flow}
We let $\mathscr C_B(\R^n)$ denote the space of curvature tensors on $\R^n$ satisfying the algebraic Bianchi identities. This is a subset of the space $S^2(\Lambda^2 \,\R^n)$ and comes equipped with a natural action of $O(n)$. We similarly define $\mathscr C_B(T_pM)$ and let $\mathscr C_{B}(TM)$ denote the vector bundle of algebraic curvature tensors over $(M, g)$. Preserved cones have played an important role in understanding higher dimensional Ricci flow. We briefly review some definitions and notations. For more information, see Chapter 7 of \cite{Bbook}. 

The curvature positivity conditions we will be concerned with here are the PIC, PIC1, and PIC2 conditions. Given a Riemannian manifold $(M, g)$, we say its curvature tensor $R$ is weakly PIC (or has nonnegative isotropic curvature) if for every $p \in M$ and every orthonormal four frame $e_1, e_2, e_3, e_4 \in T_pM$,  
\begin{equation*}
R_{1313} + R_{1414} + R_{2323} + R_{2424} - 2R_{1234} \geq 0. 
\end{equation*}
R is weakly PIC1 if for every $p \in M$, every orthonormal four frame $e_1, e_2, e_3, e_4 \in T_pM$, and every $\lambda \in [0,1]$,  
\begin{equation*}
R_{1313} + \lambda^2R_{1414} + R_{2323} + \lambda^2 R_{2424} - 2\lambda R_{1234} \geq 0.
\end{equation*}
R is weakly PIC2 if for every $p \in M$, every orthonormal four frame $e_1, e_2, e_3, e_4 \in T_pM$, and every $\lambda, \mu \in [0,1]$,  
\begin{equation*}
R_{1313} + \lambda^2R_{1414} + \mu^2R_{2323} + \lambda^2\mu^2 R_{2424} - 2\lambda\mu R_{1234} \geq 0.
\end{equation*}

The conditions above define closed cones in $\mathscr C_B(\R^n)$. We denote these cones by PIC, PIC1, and PIC2 respectively. Let $C \subset \mathscr C_B(\R^n)$ be a cone. For any $R \in C$, we let $T_RC$ denote the tangent cone. We will say $C$ has property $(\ast)$ if it satisfies the following conditions:
\begin{enumerate}
\item[(i)] $C$ is closed, convex, $O(n)$-invariant, and of full-dimension. 
\item[(ii)] For every $R \in C \setminus \{0\}$, $Q(R)$ is contained in interior of $T_RC$. 
\item[(iii)] For every $R \in C \setminus \{0\}$, $\scal(R) > 0$. 
\item[(iv)] The curvature tensor $I_{ijkl} = \delta_{ik}\delta_{jl} - \delta_{il} \delta_{jk}$ lies in the interior of $C$. 
\end{enumerate}

In particular, PIC1 and PIC2 are examples of cones satisfying property $(\ast)$. The PIC cone has every property but $(ii)$; $Q(R)$ may be contained in the boundary of the tangent cone. An important consequence of the properties above is that for any cone $C$ with property $(\ast)$ there exist constants $\theta > 0$ and $\Theta < \infty$, depending only on the cone, such that for all $R \in C$, $Q(R) - \theta \,  \scal^2I$ is contained in the interior of $T_RC$ and $|R|^2 \leq \Theta^2\, \scal^2$. See Lemma \ref{cone_prop} below.

Given a Riemannian manifold $(M, g)$, any identification between $(T_pM, g_p)$ and $(\R^n, g_{\mathrm{flat}})$ gives an identification between $\mathscr C_B(\R^n)$ and $\mathscr C_B(T_pM)$. The image of any $O(n)$-invariant subset $\mathscr C_B(\R^n)$ is independent of the choice of identification. For this reason, if $C \subset \mathscr C_B(\R^n)$ is an $O(n)$-invariant cone and $R$ is the curvature tensor on $(M, g)$ we will abuse notation by writing $R(x) \in C$, but we mean the curvature tensor at $x$ is contained in the cone $C$ after making any identification of the tangent space with $\R^n$. 

The pointwise identifications can be chosen in a consistent manner in neighborhood of $M$ via parallel transport. Using an orthonormal basis at $p$, we first identify $\mathscr C_B(\R^n)$ with $\mathscr C_B(T_pM)$. Then we use the parallel transport maps to uniquely identify $\mathscr C_B(T_pM)$ and $\mathscr C_B(T_qM)$ for any point $q$ not in the cut locus of $p$. In other words, $\mathscr C_B(TM)$ in a trivializing chart $U$ can be identified with $U \times \mathscr C_B(\R^n)$. In this way, any $O(n)$-invariant cone $C$ is defined consistently in neighborhood of $M$. Recall the Kulkarni-Nomizu product: given two symmetric (0,2)-tensors $A$ and $B$, the product defines a (0,4)-tensor $A \owedge B$ with the symmetries of an algebraic curvature tensor. In any orthonormal 4-frame, the product is defined by 
\begin{equation*}
(A \owedge B)_{pqrs} = A_{pr}B_{qs} - A_{ps}B_{qr} - A_{qr}B_{ps} + A_{qs}B_{pr}. 
\end{equation*}
The tensor in property $(iv)$ above is $I = \frac{1}{2} g \owedge g$, which in a local orthonormal frame becomes $\frac{1}{2} \id \owedge \id$. 

\section{Auxillary Results}

PIC implies a bound for the largest eigenvalue of the Ricci tensor by the scalar curvature. This is Lemma A.3 in \cite{BrePIC}. 
\begin{lemma}\label{ricci_bound}
Suppose $n \geq 5$ and $R \in \PIC$. Then $\Ric_{nn} \leq \frac{1}{2} \scal$. 
\end{lemma}
\begin{proof}
We have 
\begin{equation*}
\scal - 2 \Ric_{nn} = \sum_{j =1}^{n-1} \Ric_{jj} - \Ric_{nn} = \sum_{i, j = 1}^{n-1} R_{ijij} \geq 0, 
\end{equation*}
by summing over the PIC condition. 
\end{proof}

A curvature tensor $R$ contained in a cone $C$ satisfying property $(\ast)$ must have $Q(R)$ a uniform distance from the boundary of $T_RC$ and is uniformly bounded in norm by its scalar curvature. 

\begin{lemma} \label{cone_prop}
Suppose $n \geq 4$, and $C \subset \mathscr C_B(\R^n)$ a convex, closed cone. 
\begin{enumerate} 
\item If $C$ has property $(ii)$ above, then there exists $\theta$ such that for all $R \in C$, $Q(R) - \theta \, \scal^2 I \in T_RC$. 
\item If $C$ has property $(iii)$ above, then there exists $\Theta$ such that for all $R \in C$, then
\begin{equation*}
|R| = \sum_{i,j = 1}^n |R_{ijij}| \leq \Theta\, \scal. 
\end{equation*}
In particular, this is true for the PIC cone. 
\end{enumerate}
\end{lemma}
\begin{proof}
For the first claim, note that $Q(R)$ depends smoothy upon on $R$. If $R = 0$, the claim is clear. Suppose $R \neq 0$. Since $C$ is a cone, for $a > 0$ we have $T_{aR}C = T_RC$ as subsets of $\mathscr C_B(\R^n)$. Thus $Q(R) - \theta \, \scal(R)^2$ is contained in the interior of $T_R{C}$ if $|R|^{-2}(Q(R) - \theta \, \scal^2)$ is contained in the interior of $T_RC = T_{|R|^{-1} R}C$. So we may assume that $|R| = 1$. If the claim is false, then there exists a sequence $R_k \in C$ with $|R_k | = 1$ so that $Q(R_k) - k^{-2} \scal(R_k)^2 \in \d T_{R_k} C$. The set $\{|R| =1 \} \cap C$ is compact, so there exists a subsequence converging to $R \in C$. Moreover, $T_{R_k}C$ converges to $T_RC$ continuously, taking boundary points to boundary points. So $Q(R)$ is contained in the boundary of $T_RC$ contradicting $(ii)$. 

The second claim follows similarly. Assume $R \neq 0$. Then $\scal > 0$. If the claim is false, then there exists a sequence $R_k \in C$ such that $|R_k| = 1$, and $\scal(R_k) \leq k^{-1}$. Since $\{|R| = 1\} \cap C$ is closed, we conclude there exists $R \in C$ such that $|R| = 1$, but $\scal = 0$, contradicting $(iii)$. 

To verify the second property for PIC, note by Proposition 7.3 in \cite{Bbook}, $R \in \PIC$ implies $\scal \geq 0$ and $R \in \PIC$ with $\Ric = 0$ implies $R = 0$. If $\scal = 0$, then by Lemma \ref{ricci_bound}, the largest eigenvalue of the Ricci tensor is nonpositive. If any eigenvalue is negative, then $\scal < 0$. Since it is not, we conclude $\Ric = 0$ and hence $R = 0$. Thus PIC has property (ii). 
\end{proof}

The following lemma shows that uniform two-positivity of the Ricci tensor is preserved for curvature tensors that are uniformly PIC. A lower bound for the sum of the smallest two eigenvalues of the Ricci tensor is one of the conditions required in the definition of Brendle's preserved family of cones. 

\begin{lemma}\label{pic_ricci}
Suppose $\delta > 0$, $n \geq 5$, and $R \in \mathscr C_B(\R^n) \setminus \{0\}$ satisfies $R - \delta \, \scal \, I \in \PIC$. Then, there exists a constant $\theta_0 = \theta_0(n, \delta) > 0$ such that if $\theta \in [0, \theta_0]$ and $\Ric_{11} + \Ric_{22} \leq \theta \, \scal$, then
\begin{equation*}
\Ric(Q(R))_{11} + \Ric(Q(R))_{22} - \theta \, \scal (Q(R)) > 0. 
\end{equation*}
In particular, the condition $\Ric_{11} + \Ric_{22} \geq \theta \, \scal $ is preserved by the Hamilton ODE. 
\end{lemma}
\begin{proof}
Suppose $\Ric_{11} + \Ric_{22} \leq \theta\, \scal$. We need to show that 
\begin{equation*}
\frac{d}{dt}(\Ric_{11} + \Ric_{22} - \theta \, \scal ) =  \Ric(Q(R))_{11} + \Ric(Q(R))_{22} - \theta \, \scal (Q(R)) > 0. 
\end{equation*}
Without loss of generality we may assume $\Ric$ is diagonal and $\Ric_{11} \leq \dots \leq \Ric_{nn}$. Let $\lambda_i = \Ric_{ii}$. Then $\lambda_n \geq \frac{1}{n} \scal$ and by the Lemma \ref{ricci_bound}, $\lambda_n \leq \frac{1}{2} \scal$. Note the PIC condition implies $\scal > 0$.  Now we compute
\begin{align*}
\Ric(Q(R))_{11}+ \Ric(Q(R))_{22} &= 2 \sum_{j =1}^n \big(R_{1j1j} + R_{2j2j}\big)\lambda_j \\
& = 2 \sum_{j =1}^n \big(R_{1j1j} + R_{2j2j}\big)\big(\lambda_j - \frac{1}{2}(\lambda_1 + \lambda_2)  \big)  + \sum_{j=1}^n\big(R_{1j1j} + R_{2j2j}\big)(\lambda_1 + \lambda_2)  \\
& =  2 \sum_{j =3}^n \big(R_{1j1j} + R_{2j2j}\big)\big(\lambda_j - \frac{1}{2}(\lambda_1 + \lambda_2)  \big)  + (\lambda_1 + \lambda_2)^2 \\
& \geq 4 \delta \, \scal \sum_{j =3}^n \big(\lambda_j - \frac{1}{2}(\lambda_1 + \lambda_2)\big)\\
&\qquad \qquad + 2 \sum_{j=3}^n \big(R_{1j1j} + R_{2j2j} - 2 \delta \, \scal \big)\big(\lambda_j - \frac{1}{2}(\lambda_1 + \lambda_2)  \big) \\
& = 4 \delta \, \scal \big(\scal - \lambda_n  - \frac{n}{2}(\lambda_1 + \lambda_2)\big) \\
&\qquad \qquad + 2 \sum_{j=3}^n \big(R_{1j1j} + R_{2j2j} - 2 \delta \, \scal \big)\big(\lambda_j - \frac{1}{2}(\lambda_1 + \lambda_2)  \big) \\
& \geq  2\delta \, \scal^2 \big(1- n \theta \big) + 2 \sum_{j=3}^n \big(R_{1j1j} + R_{2j2j} - 2 \delta \, \scal \big)\big(\lambda_j - \frac{1}{2}(\lambda_1 + \lambda_2)  \big).
\end{align*}
Note $\lambda_j - \frac{1}{2}(\lambda_1 + \lambda_2) \geq 0$ for $j \geq 3$. We will show the sum is controlled by the first term. Uniformity of the PIC condition means for all distinct $p, q \in \{3, \dots, n\}$, 
\begin{equation*}
R_{1p1p} + R_{2p2p} + R_{1q1q} + R_{2q2q} \geq 4\delta\, \scal.  
\end{equation*} 
This condition implies that the expression $R_{1p1p} + R_{2p2p} - 2\delta \scal$ can be negative for at most one $p \in \{3, \dots, n\}$. We consider three cases:\\\\
\textit{Case 1:} Suppose that $R_{1p1p} + R_{2p2p} \geq 2\delta \, \scal$ for all $p \in \{3, \dots, n\}$. Then every term in the sum is nonnegative, so 
\begin{equation*}
2 \sum_{j=3}^n \big(R_{1j1j} + R_{2j2j} - 2 \delta \, \scal \big)\big(\lambda_j - \frac{1}{2}(\lambda_1 + \lambda_2)  \big) \geq 0.
\end{equation*}
\textit{Case 2:} Suppose that $R_{1p1p} + R_{2p2p} < 2\delta \, \scal$ for some $p \in \{3, \dots, n -1 \}$. Then $R_{1q1q} + R_{2q2q}  \geq 2 \delta \, \scal$ for all $q \in \{3, \dots, n\} \setminus \{p\}$ and in particular, 
\begin{equation*}
R_{1n1n} + R_{2n2n} - 2\delta \, \scal \geq 2\delta \, \scal - R_{1p1p} - R_{2p2p}.
\end{equation*}
This implies
\begin{equation*}
2\sum_{j =3}^n \big(R_{1j1j} + R_{2j2j} - 2 \delta \, \scal \big)\big(\lambda_j - \frac{1}{2}(\lambda_1 + \lambda_2) \big)  \geq  2\big(2 \delta \, \scal - R_{1p1p} - R_{2p2p}\big)\big(\lambda_{n} - \lambda_{p}\big) \geq 0
\end{equation*} 
\textit{Case 3:} Suppose that $R_{1n1n} + R_{2n2n} < 2 \delta \, \scal$. Then for all $q \in \{3, \dots, n-1\}$
\begin{equation*}
R_{1q1q} + R_{2q2q} - 2\delta \, \scal  \geq 2\delta \, \scal - R_{1n1n} -R_{2n2n}. 
\end{equation*}
This implies
\begin{align*}
2\sum_{j =3}^n \big(R_{1j1j} + R_{2j2j} -& 2 \delta \, \scal \big)\big(\lambda_j - \frac{1}{2}(\lambda_1 + \lambda_2) \big) \\
&  \geq 2\big(2\delta\, \scal - R_{1n1n} - R_{2n2n}\big)\bigg(\sum_{j=3}^{n-1} \big(\lambda_j- \frac{1}{2}(\lambda_1 + \lambda_2)\big) - \lambda_n + \frac{1}{2} (\lambda_1 + \lambda_2)\bigg) \\
& =  2\big(2\delta\, \scal - R_{1n1n} - R_{2n2n}\big)\big(\scal - 2\lambda_n -\frac{n-2}{2}(\lambda_1 + \lambda_2)\big) \\
& \geq -\big(2\delta\, \scal - R_{1n1n} - R_{2n2n}\big)(n-2)\theta \, \scal \\
& \geq -(n-2)\big(2\delta + \frac{|R_{1n1n}| + |R_{2n2n}|}{\scal}\big) \theta \, \scal^2 \\
& \geq -(n-2)\big(2\delta + C(n)\big) \theta \, \scal^2 \\
& \geq -\frac{\delta}{2}\, \scal^2
\end{align*}
where the last inequality holds for $0 \leq \theta \leq \hat \theta = \hat \theta(n, \delta)$, sufficiently small. Note we used that $|R_{1n1n}| + |R_{2n2n}| \leq C(n)\, \scal$ by Lemma \ref{cone_prop} to go from the fourth line to the fifth. 

By our casework above, for $\theta$ sufficiently small, we have
\begin{equation*}
\Ric(Q(R))_{11}+ \Ric(Q(R))_{22}  \geq   2\delta \, \scal^2 \big(1- n \theta \big) - \frac{\delta}{2} \scal^2  \geq \frac{\delta}{2} \, \scal^2
\end{equation*}
On the other hand, 
\begin{equation*}
\theta \, \scal(Q(R)) = 2 \theta\, |\Ric|^2 \leq 2 n \theta \, \lambda_n^2 \leq \frac{n}{2} \theta \, \scal^2.
\end{equation*}
Putting this together with our lower bound, we have
\begin{equation*}
\Ric(Q(R))_{11} + \Ric(Q(R))_{22} - \theta \, \scal (Q(R)) \geq \big(\frac{\delta}{2} - \frac{n}{2} \theta\big) \scal^2. 
\end{equation*}
The claim then follows for $\theta \leq \theta_0 = \theta_0(n , \delta) = \min\{\hat \theta(n, \delta) , \frac{\delta}{2n} \} > 0$.
\end{proof}

The following standard lemma is based on Hamiton's proof of the maximum principle for systems \cite{Ham86} (see also \cite{BW1}). We include it here for convenience. 
\begin{lemma}\label{bohmwilk}
Suppose $(M, g)$ is a Riemannian manifold and let $S$ be a local smooth section of $\mathscr C_B(TM)$ in a trivialization. That is, $S : U \to \mathscr C_B(\R^n)$.  Let $C \subset \mathscr C_B(\R^n)$ be a cone satisfying property $(\ast)$ and assume $S(x) \in C$ for every $x \in U$. Then for any $x \in U$ and $v \in T_xM$, $\nabla_v S(x) \in T_{S(x)}C$ and $\Delta S (x) \in T_{S(x)} C$.  
\end{lemma}
\begin{proof}
The tangent cone at interior points of $C$ is all of $\mathscr C_B(\R^n)$, so wherever $S(x)$ is contained in the interior of $C$, there is nothing to prove. Suppose $S(x) \in \d C$. The tangent cone $T_{S(x)}C$ is intersection of all half-spaces such that $S(x) \in \d H$ and $C \subset H$. Let $H \subset \mathscr C_B(\R^n)$ be any such supporting half-space of $C$. Denote the natural inner product on $\mathscr C_B(\R^n)$ by $\langle \cdot\,,\, \cdot \rangle$. Let $R_H$ be the normal vector defining $H$ so that $R \in H$ if $\langle R, R_H \rangle \geq 0$. Note that $|\langle R, R_H \rangle|$ is the distance from $R$ to $H$ in $\mathscr C_B(\R^n)$. We must show $\langle \nabla_v S(x), R_H \rangle \geq 0$ and $\langle \Delta S(x), R_H \rangle \geq 0$. Let $\gamma(t)$ be a geodesic in $M$ satisfying $\gamma(0) = x$ and $\gamma'(0) = v$. Define $f(t) = \langle S(\gamma(t)), R_H \rangle$. Then $f(t) \geq 0$ and attains a minimum at $t = 0$. Therefore 
\begin{gather*}
0 = f'(t) = \langle \nabla_v S(x), R_H \rangle,\\
0 \leq f''(t) = \langle \nabla^2_{v,v} S(x), R_H \rangle. 
\end{gather*}
Since $v$ is arbitrary, summing the second identity over an orthonormal basis at $x$ implies $\langle \Delta S(x), R_H \rangle \geq 0$. Since $H$ was arbitrary, the lemma is proven. 
\end{proof}

\section{Proof of Theorem 1.1}

\begin{proof}
This theorem is a simple extension of Theorem 2 in [MW17] using convergence results of \cite{BrendleConv} instead of \cite{BW2}. By working on the universal cover, we may assume $(M, g)$ is simply connected. Since our soliton is weakly $\PIC2$ and strictly $\PIC$, by Proposition 6.5 in \cite{BrePIC}, either $\Ric$ has a zero and $(M, g)$ splits off a line or $\Ric > 0$. If $(M, g)$ is isometric to $(N, g_N) \times \R$, then one can easily verify that with $f_N = f|_{N\times\{0\}}$, the triple $(N, g_N, f_N)$ is a complete gradient shrinking soliton of lower dimension satisfying the same assumptions as $(M, g, f)$ and additionally $\Ric_{g_N} > 0$. So it suffices to assume $\Ric > 0$. The condition weakly PIC2 implies nonnegative sectional curvature. Together with the assumption $\Ric > 0$, the aforementioned result of Munteanu and Wang implies $(M, g)$ is compact. Since $M$ is compact, it cannot split off a second line and so the argument given in Proposition 6.6 of \cite{BrePIC} implies $M$ is strictly PIC2. By the main result of \cite{BrendleConv}, the associated Ricci flow of $(M, g, f)$ converges to $S^n$, and the desired result follows. 
\end{proof}

\section{Proof of Theorem 1.3}

\begin{proof}
If $M$ is compact, then the result follows from Theorem 9 in \cite{BHS} because shrinking solitons are ancient solutions. So assume $M$ is noncompact and the soliton is normalized. Note because $M$ is not flat, $\scal > 0$. Recall normalized solitons satisfy $\scal \leq f$, and $|\nabla f|^2 \leq f$. 
Define 
\begin{equation*}
s_{\max} = \sup \{ s \in [0, 1] \, : \, \forall t \leq s,  R \in C(t)\}
\end{equation*}
By definition, $R \in C(s_{\max})$. We want to show $s_{\max} = 1$. We claim there exists some $\delta > 0$ such that for all $x \in M$, 
\begin{equation}\label{theclaim}
R - \delta \, \scal \,  I \in C(s_{\max}).
\end{equation}
The theorem follows from the claim: if it is true and $s_{\max} < 1$, then because the family of cones is continuous and contains $I$, this implies $R \in C(s_{\max} + s')$ for $s' > 0$ sufficiently small, contradicting the definition of $s_{\max}$. 

The claim \eqref{theclaim} will follow from the maximum principle. We will need a function to control the growth of curvature. In light of the fact that $\scal \leq f$, a suitable power of the soliton potential is the obvious choice. Define $\varphi : M \to \R$ by 
\begin{equation*}
\varphi = (f+n)^2 
\end{equation*}
Observe that 
\begin{align} \label{solitonpotential}
\Delta_f (f+n)^2 &= \Delta_f f^2 + 2n \Delta_f f \\
\nonumber & = 2 |\nabla f|^2 + 2f \Delta_f f + 2n\big(\frac{n}{2} - f\big) \\
\nonumber & = 2|\nabla f|^2 + n f - 2 f^2 + n^2 - 2nf \\
\nonumber & \leq n^2 - (n-2)f - 2f^2 \\
\nonumber & \leq n^2 \\
\nonumber & \leq (f+n)^2
\end{align}
where we used that $|\nabla f|^2 \leq f$ and $f \geq 0$. So we have $\Delta_f \varphi \leq \varphi$. Moreover, since $\scal \leq f$ and $f(x) \to \infty$ as $x \to \infty$, we have
\begin{equation}\label{phichoice}
\lim_{x \to \infty} \frac{\scal(x)}{\varphi(x)} = 0.
\end{equation}
Fix an arbitrary $\varepsilon > 0$ and $\delta > 0$ to be chosen later. Define a function $u : M \to \R$ by 
\begin{equation*}
u(x) = \inf \Big\{ \lambda \geq 0 \, : \, R(x)+\big(\varepsilon \varphi(x) - \delta\scal(x)\big)I + \lambda I \in C(s_{\max})\Big\}.
\end{equation*}
This function measures the distance in the direction of $I$ of the section $R + (\varepsilon \varphi - \delta\scal)I$ from the interior of the cone $C(s_{\max})$. By $O(n)$-invariance of the cone, the definition of $u$ is independent of the local orthonormal frame we use to define $I$. We claim that $u \equiv 0$ on $M$. Our choice of $\varphi$ such that \eqref{phichoice} holds implies that $u$ vanishes outside of a compact set in $M$. Suppose however $u \not \equiv 0$. Then there exists $x_0 \in M$ where $u$ attains its positive maximum $\lambda_0 > 0$. Let $U$ be a small open neighborhood of $x_0$. Define a smooth local section of $\mathscr C_B(TM)$ by
\begin{equation*}
S(x) = R(x) + \big(\varepsilon \varphi(x) - \delta\scal(x)\big)I + \lambda_0 I, 
\end{equation*}
with respect to some local orthonormal frame. In a local trivialization of $\mathscr C_B(TM)$, the section is a smooth map $S : U \to \mathscr C_B(\R^n)$.  Because $\lambda_0$ is the maximum of $u$, we have that $S(x) \in C(s_{\max})$ for all $x \in U$ and $S(x_0) \in \d C(s_{\max})$. Hence by Lemma \ref{bohmwilk}, $\Delta_f S(x_0) \in T_{S(x_0)}C (s_{\max})$. 

To derive a contradiction, we will show that $-\Delta_f S(x_0)$ is contained in the interior of $T_{S(x_0)}C(s_{\max})$. To this end, let $c = \varepsilon \varphi - \delta \scal + \lambda_0$ so that  $S = R + cI$. Then by straightforward computation 
\begin{align*}
Q(S) &= Q(R + cI) = Q(R) +2c \, \Ric \owedge \id +2nc^2I, \\
\scal(S) &= \scal + n(n-1)\big(\varepsilon \varphi - \delta \, \scal + \lambda_0\big). 
\end{align*}
Since $\varepsilon \varphi + \lambda_0 \geq 0$, this implies $\scal(S) \geq \scal - n(n-1)\delta \scal$. So if we assume $\delta \leq (2n(n-1))^{-1}$, then
\begin{equation} \label{scalofs}
\scal(S(x_0))^2 \geq \frac{1}{2}\scal(x_0)^2. 
\end{equation}
Now because $R(x_0) \in C(s_{\max})$, by definition of $\lambda_0$, we must have $\varepsilon \varphi(x_0) - \delta \, \scal(x_0) + \lambda_0 \leq 0$. In particular, 
\begin{equation}\label{cest}
|c(x_0)| = -\varepsilon \varphi(x_0) + \delta \, \scal(x_0) - \lambda_0 \leq \delta\,\scal(x_0). 
\end{equation}
Fix $\theta$ and $\Theta$ for $C(s_{\max})$ as in Lemma \ref{cone_prop}. Now we compute,
\begin{align*}
-\Delta_f S & = -\Delta_f R - \varepsilon \Delta_f\varphi I + \delta\Delta_f\scal \,I \\
& = Q(R) - R + \delta\,\scal\, I - 2\delta |\Ric|^2 I   - \varepsilon \Delta_f\varphi I  \\
& = Q(S) - 2\big(\varepsilon \varphi - \delta \, \scal + \lambda_0) \Ric \owedge \id - 2n(\varepsilon \varphi - \delta \, \scal + \lambda_0) ^2 I \\ 
& \qquad \qquad -S + \varepsilon \varphi I + \lambda_0 I - 2\delta |\Ric|^2 I - \varepsilon \Delta_f\varphi \, I   \\
& = (Q(S) - \theta\, \scal(S)^2I)  + \theta \,\scal (S)^2 I  - S + \varepsilon(\varphi - \Delta_f \varphi) I  + \lambda_0 I \\
& \qquad \qquad   - 2\delta |\Ric|^2 I - 2\big(\varepsilon \varphi- \delta \, \scal + \lambda_0) \Ric \owedge \id- 2n(\varepsilon \varphi - \delta \, \scal + \lambda_0) ^2 I 
\end{align*}
In our final expression above for $-\Delta_f S(x_0)$ the terms on the first line are contained in $T_{S(x_0)}C(s_{\max})$. We need to control the remaining terms which are not contained in the tangent cone. By \eqref{scalofs} and \eqref{cest}, at $x_0$ these terms are controlled in norm by $\scal(x_0)^2$ and hence by $\scal(S(x_0))^2$. Specifically, at $x_0$ there holds
\begin{align*}
2|\varepsilon \varphi - \delta \, \scal + \lambda_0| \big|\Ric \owedge \id\big| &\leq C(n) \,  \Theta\, \delta \, \scal^2, \\
 2n(\varepsilon \varphi - \delta \, \scal + \lambda_0)^2 |I| &\leq C(n) \, \delta \,  \scal^2, \\
 2\delta |\Ric|^2 |I|&\leq C(n) \,  \Theta^2 \delta \, \scal^2,
\end{align*}
where $C(n)$ is a constant only depending on the dimension. Thus for $\delta$ sufficiently small depending only upon $n, \theta$, and $\Theta$, the terms estimated above are dominated by the term $ \theta \,\scal (S)^2 I $, which is contained in the interior of the tangent cone. For such $\delta$, we conclude that $-\Delta_f S(x_0)$ is contained in the interior of $T_{S(x_0)}C(s_{\max})$, a contradiction. Thus we must have $u(x_0) =0$, and hence $u \equiv 0$ on $M$. In other words, $R - \delta \, \scal\, I + \varepsilon \varphi I \in C(s_{\max})$. By taking $\varepsilon \to 0$, the proof of the claim, and hence the theorem is complete. 
\end{proof}

\section{Proof of Theorem 1.4}

We begin by showing that the Ricci tensor of a uniformly PIC shrinking soliton is uniformly two-positive. 

\begin{lemma}\label{ricci_pinched}
Suppose $\delta > 0$, $n\geq 5$ and $(M, g, f)$ is a non-flat $n$-dimensional complete gradient shrinking soliton with curvature tensor satisfying $R - \delta \, \scal \, I \in \PIC$. Then the Ricci tensor of $(M, g)$ is uniformly two-positive: there exists $\theta = \theta(n, \delta) > 0$ such that 
\begin{equation*}
\Ric_{11} + \Ric_{22} \geq \theta \, \scal.
\end{equation*} 
\end{lemma}
\begin{proof}
The proof is essentially a combination of two ideas used above. First, uniformity of the PIC assumption implies a good evolution equation for the sum of the two smallest eigenvalues of Ricci. In the presence of bounded curvature, this is enough to show two-positivity for ancient solutions. Since we do not assume bounded curvature, the second input is that uniformity of the PIC assumption implies the soliton potential controls the norm of the Ricci tensor.

Assume the soliton is normalized. We assume $M$ is noncompact. The argument also works with obvious modifications if $M$ is compact and hence has bounded curvature. Choose any $\theta \in [0, \theta_0]$ as in Lemma \ref{pic_ricci}. Let $\lambda_i : M \to \R$ denote the eigenvalues of the Ricci tensor with $\lambda_1 \leq \dots \leq \lambda_n$. For $x \in M$, let $e_1, \dots, e_n \in T_xM$ be an eigenbasis of Ricci. Then at $x$, in the sense of barriers, we have 
\begin{equation*}
\Delta_f(\lambda_1 + \lambda_2 - \theta\, \scal) \leq (\lambda_1 + \lambda_2 - \theta \, \scal) - \Ric(Q(R))_{11} - \Ric(Q(R))_{22} +  \theta \, \scal (Q(R))\\
\end{equation*} 
By Lemma \ref{pic_ricci}, whenever $\lambda_1 + \lambda_2  \leq \theta \, \scal$, 
\begin{equation*}
 \Ric(Q(R))_{11} + \Ric(Q(R))_{22} -   \theta \, \scal (Q(R))> 0, 
\end{equation*}
and thus $\Delta_f(\lambda_1 + \lambda_2 - \theta \, \scal) \leq (\lambda_1 + \lambda_2 - \theta \, \scal)$, whenever $\lambda_1 + \lambda_2  - \theta \, \scal \leq 0$,  

As in the proof of Theorem \ref{cones}, let $\varphi = (f+n)^2$. Recall that $\Delta_f \varphi \leq \varphi$ and $\varphi > 0$. 
Fix an arbitrary $\varepsilon > 0$. Define a function 
\begin{equation*}
u= \lambda_1 + \lambda_2 - \theta \, \scal +  \varepsilon \varphi. 
\end{equation*} 
Then we have $\Delta_f u \leq u $ whenever $\lambda_1 + \lambda_2  - \theta \, \scal \leq 0$. 

Now by Lemma \ref{cone_prop} for some constant $C(n)$, we have 
\begin{equation*}
|\lambda _1 + \lambda _2| +\theta \, \scal \leq 2|\Ric| + \theta\, \scal \leq C(n)\scal \leq C(n)f. 
\end{equation*} Therefore 
\begin{equation*}
u \geq - |\lambda_1 + \lambda_2| - \theta \, \scal + \varepsilon f^2 \geq -C(n)f + \varepsilon f^2
\end{equation*} 
Since $f(x) \to \infty$ as $d(x, p) \to \infty$, we conclude $u \geq 0$ outside of a compact set in $M$ and that $u \geq - C$ on $M$ for some $C > 0$. Suppose the function attains its minimum at $x_0$. If $u(x_0) < 0$, then since the soliton potential is positive, we must have $(\lambda_1 + \lambda_2 - \theta \, \scal )(x_0) < 0$. So $\Delta_f u(x_0) \leq u(x_0) < 0$. This contradicts the fact that $\Delta_f u (x_0) \geq 0$ at the minimum. Therefore, $u \geq 0$ on $M$ and by taking $\varepsilon \to 0$, we have $\lambda_1 + \lambda_2 \geq \theta \, \scal$, as was claimed. 
\end{proof}

\begin{lemma}\label{furtherpinched}
Suppose $n \geq 5$ and $(M, g, f)$ is an $n$-dimensional gradient shrinking soliton with curvature tensor that is weakly PIC1. Then the curvature tensor of $(M, g)$ is weakly PIC2. 
\end{lemma}
\begin{proof}
Assume the soliton is normalized. Because $R \in \PIC1$, there exists $\Theta = \Theta(n)$ such that $|R| \leq \Theta \, \scal$. When $(M, g)$ is flat, the theorem is clear. Assume $(M, g)$ is not flat and $\scal > 0$. Fix $0 < \varepsilon \leq 1$ arbitrarily small and let $\varphi = (f+n)^2$ as in the proof of Theorem \ref{cones}. As opposed to previously, here we will use the stronger conclusion that $\Delta_f \varphi < \varphi$, which is true because $f > 0$ (see \eqref{solitonpotential}).  Define a curvature tensor on $M$ by $S = R + \varepsilon \, \scal \, I$. Note that $|S| \leq C(n) \Theta\,\scal$.  We claim there exists $\kappa < \infty$ such that 
\begin{equation}\label{goodeq}
S_{1313} + \lambda^2 S_{1414} + \mu^2 S_{2323} + \lambda^2 \mu^2 S_{2424} - 2\lambda\mu S_{1234} + \kappa\,\varepsilon\, \varphi(x) \, (1-\lambda^2)(1-\mu^2) \geq 0
\end{equation}
for all $x \in M$, $\{e_1, e_2, e_3, e_4 \} \subset T_xM$ orthonormal, and $\lambda, \mu \in [0,1]$. To see this, suppose we have $x, \{e_1, e_2, e_3, e_4\}, \lambda, \mu$ and $\kappa$ such that the inequality above does not hold. Because the last term in the inequality above is nonnegative, failure to hold implies 
\begin{equation}\label{badimplies}
S_{1313} + \lambda^2 S_{1414} + \mu^2 S_{2323} + \lambda^2 \mu^2 S_{2424} - 2\lambda\mu S_{1234}  < 0. 
\end{equation}
Rearranging terms in \eqref{goodeq} but with the inequality flipped, we have
\begin{align}\label{badeq}
\kappa \, \varepsilon \, (1- \mu^2)(1-\lambda^2) &< \varphi(x)^{-1}\Big(|S_{1313}| + \lambda^2 |S_{1414}| + \mu^2 |S_{2323}| + \lambda^2 \mu^2 |S_{2424}| + 2\lambda\mu |S_{1234}| \Big) \\
 \nonumber &\leq  C(n) \Theta \frac{\scal(x)}{\varphi(x)}.
\end{align}
As we have used several times now, the RHS is bounded on $M$. So it remains to show the coefficient of $\kappa$ on the LHS is bounded away from zero. By the $\PIC1$ assumption, for our given $\lambda$ we have 
\begin{equation*}
S_{1313} + \lambda^2S_{1414} + S_{2323} +  \lambda^2S_{2424} - 2\lambda S_{1234} \geq 2\varepsilon \, (1+\lambda^2)\, \scal 
\end{equation*}
Combining this with \eqref{badimplies} implies
\begin{equation*}
(1-\mu^2) S_{2323} + \lambda^2 (1- \mu^2) S_{2424} - 2\lambda(1 -\mu) S_{1234} \geq 2\varepsilon \,(1+\lambda^2)\, \scal 
\end{equation*}
Using $|S| \leq C(n)\, \Theta\, \scal$, this implies that at any point for which \eqref{badeq} holds, we have
\begin{equation*}
(1-\mu^2) \geq \frac{\varepsilon}{C(n) \Theta} > 0.
\end{equation*}
So $\mu$ is uniformly bounded away from $1$ by a constant depending only on $n$ and $\varepsilon$. An identical argument implies the same for $\lambda$ and this shows the coefficient of $\kappa$ is bounded away from zero in \eqref{badeq}. This completes the proof of the claim: For $\kappa$ sufficiently large depending upon of $\varepsilon$ and $n$, \eqref{badeq} cannot hold, and so \eqref{goodeq} does.

Let $\kappa_0$ be minimal such that \eqref{goodeq} holds on all of $M$. Let $x_0, \{e_1, e_2, e_3, e_4\}, \lambda, \mu$ be such that 
\begin{equation*}
S_{1313} + \lambda^2 S_{1414} + \mu^2 S_{2323} + \lambda^2 \mu^2 S_{2424} - 2\lambda\mu S_{1234} + \kappa_0\, \varepsilon\, \varphi(x_0) \, (1-\lambda^2)(1-\mu^2) = 0.
\end{equation*}
Note that we have $\lambda \neq 1$ and $\mu \neq 1$. By Proposition 7.27 in \cite{Bbook}, we have 
\begin{equation*}
Q(S)_{1313} + \lambda^2 Q(S)_{1414} + \mu^2 Q(S)_{2323} + \lambda^2 \mu^2 Q(S)_{2424} - 2\lambda\mu Q(S)_{1234} \geq 0 
\end{equation*}
Extend $\{e_1,e_2, e_3, e_4\}$ smoothly by parallel transport in a neighborhood of $x_0$ and define a smooth function $u$ on this neighborhood by 
\begin{equation*}
u(x) = S_{1313} + \lambda^2 S_{1414} + \mu^2 S_{2323} + \lambda^2 \mu^2 S_{2424} - 2\lambda\mu S_{1234} + \kappa_0\, \varepsilon\, \varphi(x) \, (1-\lambda^2)(1-\mu^2). 
\end{equation*}
Then we have $u(x_0) = 0$ and $\Delta_fu(x_0) \geq 0$. On the other hand, 
\begin{align*}
\Delta_f u(x_0) &= -Q(S)_{1313} - \lambda^2 Q(S)_{1414} - \mu^2 Q(S)_{2323} - \lambda^2 \mu^2 Q(S)_{2424} + 2\lambda\mu Q(S)_{1234} \\
& \qquad + \kappa_0\, \varepsilon\, \big(\Delta_f\varphi(x_0) - \varphi(x_0)\big) \, (1-\lambda^2)(1-\mu^2) + u(x_0)\\
& \leq  \kappa_0\, \varepsilon\, \big(\Delta_f\varphi(x_0) - \varphi(x_0)\big) \, (1-\lambda^2)(1-\mu^2)
\end{align*}
If $\kappa_0 \neq 0$, then we conclude $\Delta_f u(x_0) < 0$, a contradiction. Therefore, $\kappa_0 = 0$, and 
\begin{equation*}
S_{1313} + \lambda^2 S_{1414} + \mu^2 S_{2323} + \lambda^2 \mu^2 S_{2424} - 2\lambda\mu S_{1234}  \geq 0,
\end{equation*}
for all $x \in M$, $\{e_1, e_2, e_3, e_4 \} \subset T_xM$ orthonormal, and $\lambda, \mu \in [0,1]$. Finally, taking $\varepsilon \to 0$, we conclude that $R$ is weakly PIC2 as was claimed. 
\end{proof}

We can now prove Theorem 1.4. 
\begin{proof}
In \cite{BrePIC}, Brendle constructs two continuous families of cones, $C(b)_{b \in (0, b_{\max}]}$ and $\til C(b)_{b \in (0, \til b_{\max}]}$, both satisfying property $(\ast)$  for $n \geq 12$; see Brendle's Definition 3.1 and Definition 4.1. These families satisfy that $C(b_{\max}) = \til C(\til b_{\max})$ (Proposition 4.3) and $ \til C(b) \to C(b_{\max}) \cap \PIC1 \subset \PIC1$ continuously as $b \to 0$ (Definition 4.1). The family 
\[
\hat C(b) = \begin{cases} C(b) & b \in (0, b_{\max}] \\ \til C(b_{\max} + \til b_{\max} - b) & b \in (b_{\max}, b_{\max} + \til b_{\max}) \end{cases}
\]
therefore is a continuous family of cones with property $(\ast)$ which pinches towards $\PIC 1$. By Lemma \ref{ricci_pinched} above, we have $\Ric_{11} + \Ric_{22} \geq \theta \, \scal$. Uniform positive isotropic curvature and uniform two-positivity of the Ricci tensor imply the curvature tensor of $(M, g)$ is contained in $\hat C(b)$ some $b > 0$ small. This is readily seen after setting $T = \delta \,  \scal\,  I$ in Definition 3.1. Thus by Theorem \ref{cones}, the curvature tensor of $M$ is contained in $\hat C(b)$ for all $b$ in its definition and in particular $(M, g)$ is weakly PIC1. Now by Lemma \ref{furtherpinched}, $(M, g)$ is weakly PIC2. The theorem then follows from Theorem 1.1. 
\end{proof}

\bibliographystyle{amsalpha}
\bibliography{pic_ricci_solitons}

\sc{Department of Mathematics, Columbia University, New York, NY 10027} 

\end{document}